\documentclass[10pt, a4paper]{amsart}
\pdfoutput=1
\usepackage{amsmath,amsfonts,amssymb,amsthm,color}  
\usepackage{mathrsfs} 
\usepackage{fullpage}

\usepackage[utf8]{inputenc}
\usepackage{graphicx}

\theoremstyle{plain}
\newtheorem{theorem}{Theorem}[section]

\newtheorem{lemma}[theorem]{Lemma}
\newtheorem{proposition}[theorem]{Proposition}

\theoremstyle{definition}

\newtheorem{remark}[theorem]{Remark}

\numberwithin{equation}{section}
\newtheorem*{theorem*}{Theorem}

\newcommand{\R}{{\mathbb R}}

\newcommand{\n}{|\nabla u(t,x)|}
\newcommand{\ddd}{\, dxdt}
\def\Xint#1{\mathchoice
{\XXint\displaystyle\textstyle{#1}}%
{\XXint\textstyle\scriptstyle{#1}}%
{\XXint\scriptstyle\scriptscriptstyle{#1}}%
{\XXint\scriptscriptstyle%
\scriptscriptstyle{#1}}%
\!\int}
\def\XXint#1#2#3{{\setbox0=\hbox{$#1{#2#3}{%
\int}$ }
\vcenter{\hbox{$#2#3$ }}\kern-.6\wd0}}
\def\barint{\,\Xint -} 
\def\bariint{\barint_{} \kern-.4em \barint}
\def\bariiint{\bariint_{} \kern-.4em \barint}
\renewcommand{\iint}{\int_{}\kern-.34em \int} 
\renewcommand{\iiint}{\iint_{}\kern-.34em \int} 

\DeclareMathOperator{\dive}{div}

\providecommand{\abs}[1]{ \lvert#1  \rvert}

\providecommand{\trm}[1]{\textrm{#1}}

\title{A reverse H\"older inequality for the gradient of solutions to Trudinger's equation}
\author{Olli Saari\textsuperscript{*} and Sebastian Schwarzacher}
\date{\today}

\address{Olli Saari, Mathematical Institute, 
	University of Bonn,
	Endenicher Allee 60, 53115, Bonn,
	Germany}
	\email{saari@math.uni-bonn.de}
	
\address{Sebastian Schwarzacher, Katedra matematick\'e analy\'zy, Matematicko-fyzik\'aln\'\i\
  fakulta Univerzity Karlovy, Sokolovsk\'a 83, 186 75 Praha 8, Czech Republic}
  \email{schwarz@karlin.mff.cuni.cz}



\newcommand\blfootnote[1]{%
  \begingroup
  \renewcommand\thefootnote{}\footnote{#1}%
  \addtocounter{footnote}{-1}%
  \endgroup
}

\begin{document}

\begin{abstract}
We provide a higher integrability result for the gradient of positive solutions to Trudinger's equation (also known as the doubly non-linear equation) for the range $p\in [2,\infty)$. The estimate is achieved by refining a construction of intrinsic cylinders from the vectorial setting by incorporating estimates that are only available in the scalar case.
\end{abstract}

\maketitle

\blfootnote{\textsuperscript{*}Corresponding author}
\section{Introduction}
We study doubly non-linear parabolic partial differential equations modeled by Trudinger's equation
\begin{equation}
\label{eq:intro:equation}
\dfrac{\partial ( u^{p-1} ) }{\partial t} - \dive   ( |\nabla u|^{p-2} \nabla u ) = 0, \quad 1 < p < \infty. 
\end{equation} 
It was introduced by Trudinger in \cite{Trudinger1968}, where a scale and location invariant parabolic Harnack inequality for positive weak solutions was proved, generalizing the work of Moser \cite{Moser1964}. The equation \eqref{eq:intro:equation} is motivated for instance by its connection to the non-linear eigenvalue problem and sharp constants in Sobolev inequalities. Given any bounded domain $\Omega \subset \mathbb{R}^{n}$, one defines the first non-linear eigenvalue of the $p$-Laplacian
\begin{equation}
\label{eq:intro:eigen}
 \lambda_p =   \inf_{u \in W^{1,p}_0(\Omega) \setminus \{0\}} \frac{\int_{\Omega} |\nabla u|^{p} \, dx}{\int_{\Omega} |u|^{p} \, dx} .\end{equation}
The minimizer $v$ of the Raleygh quotient is unique up to a multiplicative constant \cite{Lindqvist1990,Lindqvist1992}. On the other hand, given any weak solution $u \in L_{loc}^{p}( 0, \infty;W_0^{1,p}(\Omega)) $ to \eqref{eq:intro:equation}, then either
\[ \lim_{t \to \infty} e^{\frac{\lambda_p}{p-1} t} u  = v , \]
or the limit, understood in $L^{p}$ sense, is identically zero \cite{Hynd2017}. 
This is the connection of the evolution equation \eqref{eq:intro:equation} and the geometric quantity \eqref{eq:intro:eigen}.

Our focus is on local regularity theory of positive weak solutions to \eqref{eq:intro:equation}. 
Despite the simple form of the parabolic Harnack estimate, 
the regularity theory of \eqref{eq:intro:equation} is not as simple as one might expect. 
Weak solutions are known to be locally H\"older continuous \cite{Vespri1992},
but to the best of our knowledge already H\"older continuity of the gradient is unknown. 
In addition, 
some fundamental questions such as the ones about 
uniqueness of solutions and validity of comparison principles 
are still partially open,
see \cite{Lindgren2019} for results and a discussion on open problems.
We also refer to \cite{Kinnunen2007},\cite{Kuusi2012}, \cite{Kuusi2012a} and \cite{Siljander2010} 
for more on regularity theory. 

In the present paper, 
we prove a new regularity result for the non-negative solutions of \eqref{eq:intro:equation}.
Our proof also applies to solutions of more general equations 
with principal part satisfying the natural growth conditions, see Section \ref{sec:solutions}.
The following theorem includes all exponents $p\in [2,\infty)$,
which extends over what can be deduced from the vectorial case in \cite{Boegelein2018}
by allowing arbitrarily large values of $p$.
\begin{theorem}
\label{thm:degenerate}
Let $n \ge 1$, $p \ge 2$ and $\Lambda \ge 1$.
Then there exist $c, \epsilon > 0$ only depending on $n$, $p$ and $\Lambda$ so that
for all space time cylinders $Q_{r,r^{p}} = I_{r^{p}} \times B_{ r}$ 
and for all non-negative solutions $u \in L_{loc}^{p}(I_{(4r)^{p}}; W_{loc}^{1,p}(B_{4r}))$
to \eqref{structure} and \eqref{EQ1}
it holds 
\[\bariint_{Q_{r,r^{p}}} |\nabla u| ^{p(1+\epsilon)} \ddd \leq c \left( \bariint_{Q_{2r,(2r)^{p}}} \left( \frac{u^{p}}{(2r)^{p}} + |\nabla u| ^{p}  \right) \ddd \right)^{\epsilon} \bariint_{Q_{2r,(2r)^{p}}} |\nabla u|  ^{p} \ddd.\]   
\end{theorem}

We briefly recall the history of the topic.
The study of higher integrability of gradients of solutions to partial differential equations 
goes back to Bojarski \cite{Bojarski1957} and quasiregular mappings in the plane. 
See also \cite{Gehring1973} for quasiconformal mappings in space.
Meyers \cite{Meyers1963} studied general linear elliptic equations,
and elliptic systems were later treated in \cite{Meyers1975}.
The first results for linear uniformly parabolic equations and systems are due to Giaquinta and Struwe \cite{Giaquinta1982}. 
Extending the results on higher gradient integrability to nonlinear parabolic equations
remained an open problem for some time 
until in their seminal paper \cite{Kinnunen2000} Kinnunen and Lewis 
managed to deal with systems of $p$-parabolic type.
Their proof relies on the method of intrinsic scaling,
which was originally introduced by DiBenedetto and Friedman \cite{DiBenedetto1985a} for other purposes.
The $p$-parabolic equation 
\begin{align}
\label{eq:plap}
\dfrac{\partial   u }{\partial t} - \dive   ( |\nabla u|^{p-2} \nabla u ) = 0, \quad 1<p< \infty
\end{align}
becomes degenerate or singular when the gradient of the solution vanishes,
which accounts for the behaviour qualitatively very different from what is seen in the linear setting.
The class of weak solutions is not closed under multiplication by constants
and consequently no proper homogeneous reverse H\"older condition can be expected to hold 
for the gradients of all solutions.

Equations of porous medium type
\[\dfrac{\partial   u }{\partial t} -  m \dive ( u^{m-1} \nabla u )  = 0, \quad m > 0 \]
exhibit a distinct difficulty as the equations become degenerate or singular 
depending on the values of the solution itself rather than its gradient. 
The higher gradient integrability was established only very recently by Gianazza and the second author 
in \cite{Gianazza2019} and \cite{Gianazza2019b}. 
These results require a very careful analysis of the covering properties 
of the intrinsic cylinders relative to the solution. 
The ideas there, in part originating from \cite{Schwarzacher2014}, 
have later also been used to study systems of porous medium type \cite{Boegelein2019} 
and global variants of the above mentioned problems \cite{Moring2019}. 

Equations of the type \eqref{eq:intro:equation} studied in this paper 
can be formally understood as equations for $v=u^{p-1}$ given by
\[\dfrac{\partial  v }{\partial t} - \frac{1}{(p-1)^{p-1}}\dive   ( v^{2-p} |\nabla v|^{p-2} \nabla v ) = 0 .\]
From this formulation it is clear 
that the equation becomes degenerate or singular 
depending on both the values of the solution and the values of its gradient. 
Prior to this work,
the higher gradient integrability was only known in the restricted range $p \in (2n/(n+2), 2n/(n-2)_{+} )$
when the dimension $n > 2$ and for all $p>1$ when the dimension $n \in \{1,2\}$.
These results can be found in \cite{Boegelein2018},
whose methods actually apply in the full generality of systems and sign-changing solutions.
The two restrictions on exponents are connected to the construction of intrinsic cylinders 
and the Sobolev embeddings. 
While the lower bound $2n/(n+2)$ also appears in the analogous results for the $p$-Laplacian~\eqref{eq:plap} (see~\cite{Kinnunen2000}), 
the upper bound does not.
Our contribution is to remove it when dealing with non-negative solutions to equations.
Whether or not the upper bound can be removed even in the vectorial setting remains an open problem.
\bigskip 
 
\noindent
\textbf{Acknowledgement.} 
This work was made possible by the generous support by DFG through the collaborative research center SFB 1060 at University of Bonn. We also acknowledge support by DFG through the Hausdorff Center for Mathematics under Germany's Excellence Strategy -- EXC-2047/1 -- 390685813 as well as the support of the research support programs of Charles University: PRIMUS/19/SCI/01 and UNCE/SCI/023. S.\ Schwarzacher thanks the support of the program GJ17-01694Y of the Czech national grant agency (GA\v{C}R). We wish to thank J.~Kinnunen for inspiring discussions and for introducing us to Trudinger's equation. In particular, we thank him for pointing out a mistake in an earlier version of this manuscript.
We also wish to thank the anonymous referee for several helpful remarks that improved the presentation of our results in the final form of this manuscript.

\section{Preliminaries}

\subsection{Notation}
We work on $\mathbb{R} \times \mathbb{R}^{n}$ where the first coordinate is called time and the remaining ones space. 
The dimension $n \geq 1$ is fixed throughout the paper. 
The symbols $c$ and $C$ refer to constants that only depend on quantities 
we do not keep track of.
Their actual value may change from expression to expression even within a single line.
Given two numbers $a,b > 0$,
such that $a \leq cb$ holds for a constant $c$ as above, 
we sometimes write $a \lesssim b$.
The binary relations $\gtrsim$ and $\sim$ are defined analogously. 
We write $\chi_{E}$ for the indicator function of a set $E \subset \R^{1+n}$.

Given a (positive locally finite Borel) measure $\mu$, 
usually given by a locally integrable weight function $\eta$, 
we denote the mean value over a Borel set $E$ of positive measure by
\[ \frac{1}{\mu(E)} \int_{E} f \, d \mu = \barint_{E} f \, d \mu = (f)_{E}^{\mu}. \]
In case $\mu$ is the Lebesgue measure, it is suppressed from the notation. 
We write $|E|$ for the Lebesgue outer measure of a set $E$. 
We use the same notation for both $n$-dimensional and $(n+1)$-dimensional Lebesgue measures.
It will always be clear from the context which one is being used.

Given $x \in \R^{n}$ and $r > 0$,
we write $B_r(x) = \{y \in \R^{n}: |x-y| < r\}$ for a Euclidean ball.
If the center $x$ and the radius $r$ are clear from the context or not important for the argument,
we suppress them from the notation.
Even the letter $B$ alone is reserved for $n$-dimensional Euclidean balls as above.
Similarly, 
given $t \in \R$ and $h > 0$,
we write $I_h(t) = (t-h/2,t+h/2)$ and omit $h$ and $t$ 
whenever we can without disturbing the reading of the argument.

Given two positive numbers $r$ and $s$, 
we write $Q_{r,s}$ for the space time cylinder $I_s \times B_r$.
We use the shorthand notation $bQ_{r,s} = Q_{br,bs}$ for concentric dilations. 
Similar notation is used for intervals and balls.

\subsection{Solutions}
\label{sec:solutions}
Let $\Lambda \ge 1$ and $p > 1$. 
Let $A : \mathbb{R}^{n+1} \times \mathbb{R} \times \mathbb{R}^{n} \to \mathbb{R}^{n}$ be a function measurable in each variable
and
satisfying the structural conditions
\begin{equation}
\label{structure}
A ( t,x,u , \xi  ) \cdot \xi \geq \Lambda^{-1} |\xi|^{p}, \quad |A ( t,x,u , \xi  ) | \leq \Lambda |\xi|^{p-1}  .
\end{equation}
We study non-negative weak solutions to the equations 
\[\dfrac{\partial u^{p-1} }{\partial t} - \dive A ( t,x,u,\nabla u )  = 0 .\]
Consider a space time cylinder $I_{r^{p}} \times B_{r}$.
A function $u \in L^{p}_{loc}(I_{r^{p}}; W^{1,p}_{loc}(B_{r}))$ is a weak solution if 
\begin{equation}
\label{EQ1}
\iint  A ( t,x,u,\nabla u ) \cdot \nabla \varphi \ddd- \iint u^{p-1} \cdot \partial_t \varphi\ddd = 0
\end{equation}
holds for all test functions $\varphi \in C_c^{\infty}(I_{r^{p}} \times B_{r})$. 
If the left hand side is non-positive for all non-negative test functions, 
$u$ is said to be a subsolution. 
If the left hand side is non-negative for all non-negative test functions, 
$u$ is said to be a supersolution. 
By a standard apporximation argument,
we can use compactly supported Sobolev functions as test functions 
as opposed to smooth functions.
In what follows, 
$u$ will always denote a positive weak solution to the equation.
As our main result is inherently local,
we do not refer to the the cylinder $I_{r^{p}} \times B_r$ anymore
but work as if it were the full space time.
 
\subsection{Time derivative}
\label{sec:mollification}
As the solutions are not required to exhibit any a priori differentiability in the time variable,
there is a technical difficulty when trying to use the solution itself as a test function. 
To overcome this problem, 
we can use a mollified version of the equation as an intermediate step in the proof. 
For a smooth and even function $\chi_{[-4^{-1},4^{-1}]} \leq \zeta \leq \chi_{[-2^{-1},2^{-1}]}$
and $\epsilon > 0$,
we set for any locally integrable $\varphi$
\[\varphi_\epsilon(t,x) := \int_{\mathbb{R}} \varphi(s,x) \zeta \left( \frac{t-s}{\epsilon} \right) \, \frac{ds}{\epsilon}  . \]
Using $\varphi_{\epsilon}$ with $\varphi \in C_c^{\infty}(\R^{1+n})$ as a test function, 
we can convert the equation \eqref{EQ1} to
\begin{align*}
\iint  A ( t,x,u,\nabla u )_{\epsilon} \cdot \nabla \varphi \ddd- \iint (u^{p-1})_{\epsilon} \cdot \partial_t \varphi\ddd &= 0 
\end{align*}
This together with the correct use of the mollified solution as a test function can then be used 
to justify certain computations to follow
that we decided to keep at formal level for the sake of better readability. 

\subsection{Auxiliary inequalitites}
Given a Borel probability measure $\mu$ and $q\in [1,\infty)$, the mean value $(f)_{E}^{\mu}$ always satisfies
\[  \bigg(\int_{E} |f -(f)_{E}^{\mu}|^q \, d \mu\bigg)^\frac1q   \leq 2 \inf_{c}  \bigg(\int_{E} |f -c |^q \, d \mu\bigg)^\frac1q.\] 
We refer to this fact in its various forms as the \textit{best constant property} of the mean value and it can be used to justify all occasions where we change a constant inside a mean oscillation integral as the one above.

In order to effectively carry out certain iterative arguments, 
we use the following lemma, 
which can be found as Lemma 6.1 on page 191 of \cite{Giusti2003}.
\begin{lemma}
\label{lemma:iterationlemma}
Let $A,B, \alpha > 0$ and $\delta \in (0,1)$.
Let $0< r < R < \infty$ and let $Z$ be a bounded and positive function 
such that for all $r \leq \rho_1 < \rho_2\leq R$ it holds
\[Z(\rho_1) \leq A(\rho_2 - \rho_1)^{-\alpha} + B  + \delta Z(\rho_2) .\]
Then
\[Z(r) \leq c(\alpha,\delta) \left( A(R-r)^{-\alpha} + B \right) .\]
\end{lemma}

\subsection{Positivity}
When it comes to Harnack estimates, 
there is some discrepancy in the positivity assumptions imposed on solutions in the literature, 
and we clarify that point here. 
For our estimates, we need not use the full strength of the Harnack inequality, 
but a much weaker supremum estimate for subsolutions will do. 
In particular, the use of H\"older continuity and strict positivity can be avoided. 

We recall the following observation.
Given $\epsilon > 0$ and a positive subsolution $u$, 
we let $u_ \epsilon = \max (u,\epsilon)$. 
Consider a test function $\eta \in C_c^{\infty}(\R^{1+n})$.
Following the proof of Lemma 1.1 of Section II in \cite{DiBenedetto1993a}, 
we see that the choice of test function $\varphi = \eta (u_{\epsilon} - \epsilon) (u_\epsilon - \delta )^{-1} $ with $\delta \in (0,\epsilon)$ gives (using an approximation argument as in Subsection~\ref{sec:mollification})
\begin{align*}
0	&\geq \iint  A ( t,x,u,\nabla u )  \cdot \nabla \varphi \ddd - \iint u^{p-1} \cdot \varphi_t \ddd \\
	&= \iint_{\{ u > \epsilon\}}   \frac{u_\epsilon - \epsilon}{u_\epsilon- \delta}  A ( t,x,u_\epsilon ,\nabla u_\epsilon  )  \cdot \nabla \eta  \ddd  + \iint   \frac{\eta(   \epsilon -\delta )}{(u_\epsilon - \delta)^{2}}  A ( t,x,u ,\nabla u  )  \cdot \nabla u_\epsilon  \ddd  \\
	 & \qquad - \iint  u^{p-1} \left( \eta_t \frac{u_\epsilon - \epsilon}{u_\epsilon - \delta}  +  \frac{\eta(     \epsilon - \delta) }{(u_{\epsilon} -\delta)^{2}} (u_\epsilon)_t \right)   \ddd \\
	 &\geq \iint_{\{ u > \epsilon\}}   \frac{u_\epsilon - \epsilon}{u_\epsilon- \delta}  A ( t,x,u_\epsilon ,\nabla u_\epsilon  )  \cdot \nabla \eta  \ddd   - \iint  u_\epsilon^{p-1} \left( \eta_t \frac{u_\epsilon - \epsilon}{u_\epsilon - \delta}  +  \frac{\eta(     \epsilon - \delta) }{(u_{\epsilon} -\delta)^{2}} (u_\epsilon)_t \right)   \ddd \\
	 & \longrightarrow \iint  A ( t,x,u_\epsilon ,\nabla u_\epsilon )  \cdot \nabla \eta \ddd - \iint u_{\epsilon}^{p-1} \cdot \eta_t \ddd
\end{align*}  
as $\delta \to \epsilon$. 
We could extend the domain of integration on the last line
from $\{ u > \epsilon\}$ to $\R^{1+n}$ because $|A(t,x,u_\epsilon,\nabla u_\epsilon)| \leq \Lambda |\nabla u_\epsilon|^{p-1}$ by \eqref{structure} and $\nabla u_{\epsilon} = 0$
almost everywhere in $ \{ u \le \epsilon\} $.
Hence the truncated function $u_\epsilon$ is a subsolution
that is bounded from below by $\epsilon > 0$.

The following lemma is stated in \cite{Kinnunen2007} for subsolutions $u$
with $u \ge \epsilon > 0$
but approximating by truncations as above, 
we can use it for all non-negative subsolutions $u \ge 0$.
\begin{lemma}[Lemma 5.1 \cite{Kinnunen2007}]
\label{lemma:harnack}
Consider weak subsolutions to \eqref{structure} and \eqref{EQ1}.
There exist positive constants $C = C(n,p, \Lambda)$ and $\theta(n,p)$ 
such that for all $\sigma \in (1,2)$, all $s > 0$
and all weak subsolutions $u \geq 0$ in
\[ Q_{2R,(2R)^{p}} =   I_{(2R)^{p}} \times B_{2R}   \] 
it holds 
\[ \sup_{z \in Q_{R,R^{p}}} u(z) \leq \left( \frac{C}{(\sigma - 1)^{\theta}} \bariint_{I_{(\sigma R)^{p}} \times  B_{\sigma R}} u^{s} \ddd \right)^{1/s} . \] 
\end{lemma}

\subsection{Energy estimate}
Next we derive an energy estimate for the solution.
The proof is standard, 
but we repeat the short argument for completeness.

\begin{lemma}[Caccioppoli estimate]
There is a constant $C = C(n,p,\Lambda)$ 
such that for all non-negative weak solutions $u$,
for all $1\le a < b \le 2$,
and for all space time cylinders $Q_{r,r^{p}} = I_{r^{p}} \times B_{r}$ 
\begin{equation} 
\label{eq:cacc2}
\sup_{s \in I_{ar^{p}}} \barint_{ \{s\} \times B_{ar}}  \left( \frac{u}{r} \right)^{p} \, dx + \bariint_{Q_{ar,ar^{p}}} |\nabla u|^{p} \ddd 
 \leq \frac{C}{ (b-a)^{p} } \bariint_{Q_{br,br^{p}}} \left( \frac{ u}{r} \right)^{p} \ddd .
\end{equation}
\end{lemma}

\begin{proof}
As the structure of the equation \eqref{EQ1} is scaling and translation invariant,
it suffices to prove the estimate in the case $Q = (-2^{-1},2^{-1}) \times B(0,1)$.
Consider a smooth function $\eta$ with 
\[
\chi_{aQ} \leq \eta \leq \chi_{bQ} ,\qquad |\partial_t \eta| + |\nabla \eta| \le  2 (b-a)^{-1}.
\]
Let $s \in (-a,a)$ and $h \in (0,a-s)$.
Let $\tau$ be the piecewise linear function of one variable with
\[
\chi_{(-\infty, s]} \leq \tau \leq \chi_{(-\infty, s+h]}, \qquad |\tau'| \leq h^{-1}.
\]
We use $\varphi = \eta^{p} \tau u$ as a test function in \eqref{EQ1}.

We start by estimating the first term on the left hand side of the claimed estimate.
Note first that 
\[
u^{p-1}  \varphi_t = u^{p-1} u_t \eta^{p} \tau + u^{p }  ( \eta^{p} \tau)_t = p^{-1}(u^{p})_t \eta^{p} \tau  + u^{p} (\eta^{p} \tau)_{t}.
\]
Integrating by parts in the time variable 
and using the definition of $\tau$  
\begin{multline*}
- \int_{-b}^{s+h} \int  u^{p-1}  \varphi_t \ddd 
	=  - \left(1 - \frac{1}{p} \right) \int_{-b}^{s+h} \int  u^{p} ( \eta^{p} \tau )_t   \ddd \\
	\geq \left(1- \frac{1}{p}\right) \frac{1}{h} \int_{s}^{s+h}\int  u^{p} \eta^{p} \ddd - \frac{2(p-1)}{b-a} \iint_{bQ} u^{p} \eta^{p-1} \ddd .
\end{multline*}
By the equation \eqref{EQ1} and the structural condition \eqref{structure}
\begin{multline}
\label{eq:caccproof1}
- \int_{-b}^{s+h} \int u^{p-1} \partial_t \varphi \ddd 
	=  - \iint A(t,x,u,\nabla u) \cdot (   \eta^{p} \tau \nabla u  + u \nabla (\eta^{p} \tau)  ) \ddd \\
	\leq - \frac{1}{\Lambda} \iint |\nabla u |^{p}\eta^p \tau \ddd +  \frac{2 p \Lambda}{b-a} \iint u |\nabla u|^{p-1} \eta^{p-1} \tau \ddd .
\end{multline}
Applying Young's inequality to the second term,
we bound it by 
\[
    \frac{2^{p} p^{p-1} \Lambda^{2p-1}}{ (b-a)^{p}} \iint_{bQ} u^{p} \tau \ddd + \frac{(p-1)}{p \Lambda}  \iint |\nabla u|^{p} \eta^{p} \tau \ddd  .
\]
Putting the estimates together we obtain 
\begin{multline*}
\left(1- \frac{1}{p}\right) \frac{1}{h} \int_{s}^{s+h}\int  u^{p} \eta^{p} \ddd 
+ \frac{1}{p \Lambda}\iint |\nabla u|^{p} \eta^{p} \tau \ddd  \\
\le \frac{2(p-1)}{b-a} \iint  u^{p} \eta^{p-1} \ddd + \frac{2^{p} p^{p-1} \Lambda^{2p-1}}{ (b-a)^{p}} \iint_{bQ} u^{p} \tau  \ddd .
\end{multline*}
Sending $h \to 0$ and taking the supremum over $s \in I_{ar^{p}}$,
we conclude the bound for the first term on the left hand side of \eqref{eq:cacc2}.
The second term is clear by sending $s \to a$.
\end{proof}

\begin{remark} 
By essentially the same proof,
we see that there is $C= C(n,p,\Lambda)$
such that given a general cylinder $Q_{r,s} = I_s \times B_r$ and $1\le a <b \le 2$,
then for all non-negative solutions $u$
\begin{equation}
\label{eq:cacc3}
\sup_{\tau \in I_{as} } \int_{\{ \tau \} \times B_{ar} }  \frac{u^{p}}{s}  \, dx + \bariint_{Q_{ar,as}} |\nabla u|^{p} \ddd 
 \leq \frac{C}{ b-a} \bariint_{Q_{br,bs}}  \frac{ u | \nabla u |^{p-1}}{r}  \ddd + \frac{C}{ b-a} \bariint_{ Q_{br,bs}}  \frac{ u ^{p} }{s}  \ddd.
\end{equation}
We argue as before, 
but instead of applying Young's inequality to the right hand side of \eqref{eq:caccproof1},
we accept the second term to the right hand side of \eqref{eq:cacc3}.
This form of the energy estimate will also be needed later.
\end{remark}

\begin{lemma}
\label{lemma:slice_average}
Let $u$ be a non-negative weak solution to \eqref{EQ1}. 
Let $\eta \in C_{c}^{\infty}(\mathbb{R}^{n})$. 
Then it holds
\begin{align*}
\left \lvert \int_{\mathbb{R}^{n}} u(s,x)^{p-1} \eta(x) \, dx - \int_{\mathbb{R}^{n}} u(x,t)^{p-1} \eta(x) \, dx \right \rvert \leq \Lambda \int_{s}^{t} \int_{\mathbb{R}^{n}} | \nabla u|^{p-1} |\nabla \eta| \ddd  .
\end{align*}
\end{lemma}

\begin{proof}
Take the piecewise linear $\tau_h$ such that $\chi_{(s,t)} \leq \tau_h \leq \chi_{(s-h,t+h)}$ with $|(\tau_h)'| \leq h^{-1}$. Then the left hand side of the claimed inequality equals
\[
 \lim_{h \to 0} \left \lvert \iint u^{p-1} \partial_t( \tau_h \eta ) \ddd \right \rvert
	 = \lim_{h \to 0}  \left \lvert \int_{s}^{t} \int_{\mathbb{R}^{n}} A(t,x,u,\nabla u) \cdot \nabla ( \tau_h \eta) \ddd \right \rvert 
	\leq \Lambda \int_{s}^{t} \int_{\mathbb{R}^{n}} | \nabla u|^{p-1} |\nabla \eta| \ddd.
\]
\end{proof}

\section{Intrinsic cylinders}
Consider a cylinder $Q_{\delta r,r^{p}}$ with $\delta \in (0,1]$. 
Fix a number $K \ge 1$.
Given a fixed value of $\alpha \ge 1$, 
we call a cylinder $K$-sub-intrinsic if
\begin{align}
\label{eq:modtousual_1}
\frac{ \displaystyle \left( \bariint_{ Q_{2 \delta r,(2r)^{p}} }  \left( \frac{u}{ \delta r} \right)^{\alpha p} \ddd \right)^{\frac{1}{\alpha p}} }{\displaystyle \left( \bariint_{ Q_{\delta r,r^{p}} }  |\nabla u|^{p} \ddd \right)^{\frac{1}{ p}} } \leq K \delta^{- \frac{p}{p-2}} 
\end{align}
and $K$-super-intrinsic if
\begin{align}
\label{eq:modtousual_2}
K \frac{ \displaystyle \left( \bariint_{ Q_{\delta r,r^{p}} }  \left( \frac{u}{ \delta r} \right)^{\alpha p} \ddd \right)^{\frac{1}{\alpha p}} }{\displaystyle \left( \bariint_{ Q_{ 2\delta r,(2r)^{p}} }  |\nabla u|^{p} \ddd \right)^{\frac{1}{ p}} } \geq \delta^{- \frac{p}{p-2}}  \quad \trm{or} \quad \delta^{- \frac{p}{p-2}}  \leq K .
\end{align}
A cylinder that is both super-intrinsic and sub-intrinsic is said to be intrinsic.
A general cylinder $Q_{\delta r, r^{p}}$ may or may not satisfy these conditions for a pair of values $\alpha$ and $K$.
The argument in \cite{Boegelein2018} is based on constructing a cover of the level set of the gradient of the solution 
so that every cylinder in the cover satisfies the conditions above with $\alpha = 1$. 
We carry out the corresponding construction for larger values of $\alpha$.

\begin{proposition}
\label{Prop:rhi_intrinsic}
Fix $\delta$, $r$ and a cylinder $Q_{\delta r, r}$. 
Let $\alpha \ge 1$ and assume that one of the alternatives in  \eqref{eq:modtousual_2} holds.
Then for a constant $C = C(n,p,\Lambda,\alpha,K)$  
\begin{equation*}
\sup_{Q_{\delta r, r^{p}}}  u  \leq C \left( \bariint_{Q_{ 2  \delta r, ( 2  r)^{p} }} u^{p} \ddd \right)^{1/p} .
\end{equation*}
\end{proposition}
\begin{proof}
We cover $Q_{\delta r, r^{p}}$ by cylinders $Q_i$ of dimensions $(\delta r)^{p} \times \delta r$. 
We let $1 < \gamma < \beta <2$.
as we aim for applying the iteration Lemma \ref{lemma:iterationlemma},
we can let all the constants appearing on the right hand side of the following estimates grow polynomially on $(\gamma -1)^{-1}(\beta - \gamma)^{-1}$. 

As $\gamma \in (1,2)$,
by Lemma \ref{lemma:harnack} 
\begin{equation*}
\sup_{Q_{\delta r, r^{p}}}  u = \sup_{i} \sup_{Q_i} u \lesssim   \sup_{t \in I_{ (\gamma r)^{p}  }} \left( \barint_{B_{\delta \gamma  r}} u(t,x)^{p} \, dx \right)^{1/p}.
\end{equation*}
By the 
Caccioppoli inequality for general cylinders~\eqref{eq:cacc3} 
\[
\sup_{t \in I_{ (\gamma r)^{p}  }}  \barint_{B_{\delta \gamma  r}} \frac{ u(t,x)^{p}}{r^{p}} \, dx 
+ \bariint_{Q_{\delta \gamma r, (\gamma r)^{p}}} |\nabla u|^{p} \ddd 
	\lesssim \bariint_{Q_{\beta \delta r, (\beta  r)^{p}}} \frac{u |\nabla u|^{p-1}}{\delta r} \ddd
	+ \bariint_{Q_{\beta \delta r, (\beta r)^{p}}}  \frac{u^{p}}{r^{p}} \ddd
\]
We assume that 
one of the alternatives in \eqref{eq:modtousual_2} holds.
We start with the case when the first alternative in \eqref{eq:modtousual_2} holds.

First, note that as $1 < \beta < 2$,
we know that $Q_{\beta \delta r, (\beta r)^{p}} \subset Q_{ 2 \delta  r, (2 r)^{p}}$
and consequently \eqref{eq:modtousual_2} implies 
\begin{equation*}
 \frac{ \displaystyle \left( \bariint_{ Q_{\delta \beta r, (\beta r)^{p}} }  \left( \frac{u}{ \delta r} \right)^{\alpha p} \ddd \right)^{\frac{1}{\alpha p}} }{\displaystyle \left( \bariint_{ Q_{\delta \beta r, (\beta r)^{p}} }  |\nabla u|^{p} \ddd \right)^{\frac{1}{ p}} } \gtrsim \delta^{- \frac{p}{p-2}} .   
\end{equation*}
Using this as an upper bound for $\delta^{-p}$,
we estimate 
\begin{multline*}
	\frac{1}{\delta^{p}} \bariint_{Q_{\beta \delta r, (\beta r)^{p}}} \frac{u |\nabla u|^{p-1}}{\delta r} \ddd \\
	\lesssim  \left( \bariint_{ Q_{\delta \beta r, (\beta r)^{p}} }  \left( \frac{u}{ \delta r} \right)^{\alpha p} \ddd \right)^{\frac{p-2}{\alpha p}} 
	\left( \bariint_{ Q_{\delta \beta r, (\beta r)^{p}} }  |\nabla u|^{p} \ddd \right)^{- \frac{p-2}{ p}}
	\bariint_{Q_{\beta \delta r, (\beta r)^{p}}} \frac{u |\nabla u|^{p-1}}{\delta r} \ddd .
\end{multline*}
Applying H\"older's inequality to the third factor, 
we get an upper bound by 
\[
\left( \bariint_{ Q_{\beta \delta r, (\beta r)^{p}} }  \left( \frac{u}{ \delta r } \right)^{\alpha p} \ddd \right)^{\frac{p-2}{\alpha p}}  \left( \bariint_{ Q_{\beta \delta r, (\beta r)^{p}}} |\nabla u|^{p} \ddd \right)^{\frac{1}{p}} \left( \bariint_{Q_{\beta \delta r,(\beta r)^{p} }} \frac{u^{p}}{(\delta r)^{p}} \ddd \right)^{\frac{1}{p}}.
\]
Applying Young's inequality with $\epsilon$,
we bound this by  
\[
\epsilon \left[ \left( \bariint_{ Q_{\beta \delta r, (\beta r)^{p}} }  \left( \frac{u}{ \delta r } \right)^{\alpha p} \ddd \right)^{\frac{p-2}{\alpha p} }  \left( \bariint_{ Q_{\beta \delta r, (\beta r)^{p}}} |\nabla u|^{p} \ddd \right)^{\frac{1}{p}}\right]^{\frac{p}{p-1}}
+ C_\epsilon   \bariint_{Q_{\beta \delta r,(\beta r)^{p} }} \frac{u^{p}}{(\delta r)^{p}} \ddd .
\]
Bounding $L^{\alpha}$ norm by $L^{\infty}$ norm, 
simplifying the exponents and applying Young's inequality again,
we bound the quantity inside the square bracket by
\[
 \frac{p-2}{p-1} \sup_{ Q_{\beta \delta r, (\beta r)^{p}} }  \frac{u^{p}}{ (\delta r)^{p} } + \frac{1}{p-1} \bariint_{ Q_{\beta \delta r, (\beta r)^{p}}} |\nabla u|^{p} \ddd .
\]
Multiplying through by $\delta^{p}$,
we have concluded 
\[
 \bariint_{Q_{\beta \delta r, (\beta r)^{p}}} \frac{u |\nabla u|^{p-1}}{\delta r} \ddd
 \le \epsilon \left(  \sup_{ Q_{\beta \delta r, (\beta r)^{p}} } \frac{u^{p}}{r^{p}}  + \bariint_{ Q_{\beta \delta r, (\beta r)^{p}}} \delta^{p} |\nabla u|^{p} \ddd  \right) + C \bariint_{Q_{\beta \delta r,(\beta r)^{p} }} \frac{u^{p}}{r^{p}} \ddd
\]
so that in particular
\begin{multline*}
\sup_{Q_{\delta r, r^{p}}}
 u(t,x)^{p} 
+ \bariint_{Q_{\delta \gamma r, (\gamma r)^{p}}} \delta^{p} r^{p} |\nabla u|^{p} \ddd \\
\leq \epsilon \left(  \sup_{ Q_{\beta \delta r, (\beta r)^{p}} } u^{p}  + \bariint_{ Q_{\beta \delta r, (\beta r)^{p}}} r^{p}\delta^{p} |\nabla u|^{p} \ddd  \right) + C \bariint_{Q_{\beta \delta r,(\beta r)^{p} }} u^{p} \ddd
\end{multline*}
where the constant $C$ grows polynomially in $\gamma - \beta$.
This estimate is amenable for an application of the iteration lemma \ref{lemma:iterationlemma}
so the proof is complete if $Q_{\delta r, r^{p}}$ satisfies the first condition in \eqref{eq:modtousual_2}.

It remains to study the case when the second alternative in \eqref{eq:modtousual_2} holds instead of the first one.
In that case we simply conclude 
\[ \bariint_{Q_{\beta \delta r, (\beta r)^{p}}} \frac{u |\nabla u|^{p-1}}{\delta r} \ddd  
	\lesssim   \bariint_{Q_{\beta \delta r, (\beta r)^{p}}} \frac{u |\nabla u|^{p-1}}{r} \ddd \]
so that by Young's inequality
\[\sup_{Q_{\delta r, r^{p}}}  \frac{ u^{p}}{r^{p}} +   \bariint_{ Q_{\delta r,r^{p}} }  |\nabla u|^{p} \ddd \\
	\le \epsilon \left( \sup_{Q_{\beta \delta r, (\beta r)^{p}}} \frac{u^{p}}{r^{p}}  +  \bariint_{ Q_{\delta r,r^{p}} } |\nabla u|^{p} \ddd \right) + C \bariint_{Q_{\beta r, (\beta r)^{p}}} \frac {u^{p}}{ r^{p}} \ddd	.\]
Hence the proof is complete by Lemma \ref{lemma:iterationlemma}.
\end{proof} 

\section{Construction of a differentiation basis}
We use the intrinsic scaling given by
\[\frac{s}{r^{p}} = \mu^{p-2}, \quad \mu = \frac{1}{\lambda} \left( \bariint_{Q_{r,s}}  \left( \frac{u}{ r} \right)^{\alpha p} \ddd \right)^{\frac{1}{\alpha p}}  \]
as a model for the construction of a basis of almost intrinsic cylinders. 

As before, we keep $p > 2$ fixed throughout the section.
We also fix $\alpha > (n/2)(1-p/2)$.
Let $R > 0$ and consider a cylinder $Q_{R,R^{p}} = I_{R^{p}} \times B_R$.
Let $C_{0}$ be a large constant whose exact value we will fix later.
Consider a non-negative function $u \in L^{\alpha p}_{loc}(\R^{1+n})$,
and consider numbers $\lambda \ge C_0 \lambda_0$ where
\[\lambda_0 = 1 + \left(  \bariint_{Q_{4R,(4R)^{p}}} \left( \frac{u}{R} \right)^{\alpha p} \ddd \right)^{\frac{1}{\alpha p}} . \]

Given $z = (t,x) \in Q_{2R,(2R)^{p}}$ and $\rho \in (0,R]$, 
define 
\begin{equation}
\label{eq:def:def}
d_z( \rho) = \min_{r \in [\rho, R]} \sup \left \lbrace \delta \in (0,1] : \left( \bariint_{ Q_{\delta r, r^{p}} }  \left( \frac{u}{\delta r} \right)^{\alpha p} \ddd \right)^{\frac{1}{\alpha p}} \leq \delta^{- \frac{p}{p-2}} \lambda  \right \rbrace  .  
\end{equation}
When $z$ is fixed and no confusion can arise, we drop it from the notation. 
The minimum is taken in order to make $\rho \mapsto d_z(\rho)$ an increasing function. 
As it is not strictly increasing, we cannot directly invert it. 
We define instead for $\eta \in (0,1]$.
\begin{equation}
\label{eq:rhomax}
\rho_{max}(z, \eta ) = \max \{ r : d_z( r)  = \eta  \} . 
\end{equation}
If no confusion can arise,
we drop $z$ from the notation.
It either holds 
\begin{equation}
\label{eq:rhomaxprop}
\left( \bariint_{Q_{ \eta \rho_{max}(\eta), \rho_{max}(\eta)^{p}} }  \left( \frac{u}{\eta r} \right)^{\alpha p} \ddd \right)^{\frac{1}{\alpha p}} = \eta^{- \frac{p}{p-2}} \lambda 
\end{equation}
or $\eta = 1$. 

For brevity, we denote $S(z,\rho) :=  Q_{d_z(\rho) \rho, \rho^{p}}(z)$. 
Again, we drop $z$ from the notation whenever convenient. 
We will use the sets $S(z,\rho)$ as a basis for the covering argument to follow.
If the number $\lambda$ is taken to be the appropriate mean value of the gradient of $u$,
then the cylinders $S(z,\rho)$ will be sub-intrinsic in the sense of \eqref{eq:modtousual_1}.

\begin{proposition}
\label{prop:properties_deformation}
Let $p$, $\alpha$, $C_0$, $R$, $Q_{R,R^{p}}$ and $u$ as above be given.
Then 
\begin{enumerate}
\item[(i)] For all $z \in Q_{2R,(2R)^{p}}$, it holds $d_z( \rho) > 0$. 
\item[(ii)] For all $z \in Q_{2R,(2R)^{p}}$, $d_z(\rho)$ is a continuous and increasing function in $\rho$. 
\item[(iii)] If $0<\rho \le s \le R$, then for all $z \in Q_{2R,(2R)^{p}}$ it holds $ d_z(\rho) \le  d_z(s)$ and 
\[d_z( s ) \leq \left( \frac{s}{\rho} \right)^{\frac{(n+p \alpha + p)(p-2)}{2 \alpha p - n(p-2)}} d_z( \rho).\]
In particular $S(z,\rho) \subset S(z,s)$.
\item[(iv)] Let $1 < N \le 4^{-1} (4C_0)^{(\alpha p)/(n+p+\alpha p)} $. If $\rho \in [R/N,2R]$, 
then for all $z \in Q_{2R,(2R)^{p}}$
\[
d_z(\rho)  = 1.
\]
\item[(v)] There are constants $c_i = c_i(n,p,\alpha)$, $i \in \{1,2\}$, 
so that if $z,w \in Q_{2R,(2R)^{p}}$, $0< r \le R$ and $S(z, r ) \cap S(w,r) \neq \varnothing$, then 
\[
S(z,r) \subset S(w, c_1 r )
\]
and 
\[
 \frac{1}{c_2} d_z(r) \le d_w(r) \le c_2 d_z(r) .
\]
\end{enumerate}
\end{proposition}

\begin{proof}
We start with the first item.
Simplifying the expression defining $d_z( \rho)$ in \eqref{eq:def:def},
we write the relevant left and right hand sides as 
\begin{equation}
	\label{intrinsic:add:1}
L(z,r,\delta) = \left( \frac{1}{|Q_{r, r^{p}}|} \iint_{ Q_{\delta r, r^{p}} }  \left( \frac{u}{ r} \right)^{\alpha p} \ddd \right)^{\frac{1}{\alpha p}} \leq \delta^{ \frac{n}{\alpha p}  - \frac{2}{p-2}} \lambda 
	= R(\delta).
\end{equation}
Now $L(z,r,\cdot)$ is increasing and $L(z,r,\delta) \to 0$ as $\delta \to 0$.
The right hand side $R(\cdot)$ is decreasing and $R(\delta) \to \infty$
as $\delta \to 0$
because
\[\beta := \frac{n(p-2) - 2 \alpha p}{\alpha p (p-2)} < 0.\]
This is the key point which imposes the lower bound on $\alpha$.
Hence there is a positive $\delta > 0$ so that $L(z,r,\delta) < R(\delta)$
and consequently $d_z(r) > 0$.

To verify the second item,
we first consider 
\begin{equation}
\label{intrinsic:add:2}
\tilde{d}_z(\rho) = \sup\{ \delta \in (0,1] : L(z,\rho,\delta) \le R(\delta) \}.
\end{equation}
We first show that $\tilde{d}(\rho)$ (we suppress $z$ from now on) is continuous.
Take any $r \in (0,R]$. 
We first show that $\limsup_{\rho \to r } \tilde{d}(\rho) \leq \tilde{d}(r)$.
This is satisfied by definition in case $\tilde{d}(r)=1$.
In the other case, when $\tilde{d}(r)<1$, then (again by definition), 
for all $1>\delta > \tilde{d}(r)$ the condition \eqref{intrinsic:add:1} is violated
so that $L(z,r,\delta) > R(\delta)$.
As $L(z,\rho,\delta)$ is continuous in $\rho$,
it holds $L(z,\rho,\delta) > R(\delta)$ for $|r-\rho|$ small enough.
For such $\rho$ we then have $\delta > \tilde{d}(\rho)$ 
and consequently $\limsup_{\rho \to r } \tilde{d}(\rho) \leq \tilde{d}(r)$.
To prove $\liminf_{\rho \to r } \tilde{d}(\rho) \ge \tilde{d}(r)$,
suppose for contradiction that
$ \liminf _{\rho \to r} \tilde{d}(\rho) = \theta \tilde{d}(r)$ for some $\theta < 1$. 
Then there is a sequence $\rho_i \to r$ 
so that $\tilde{d}(\rho_i) = \theta_i  \tilde{d}(r) $ for all $i$ and $\theta_i \to \theta < 1$. 
Moreover, 
as $L$ is continuous and increasing and $R$ continuous and decreasing,
we see that the supremum in \eqref{intrinsic:add:2} is a maximum.
Then 
\begin{multline*}
 \tilde{d}(r)^{- \frac{p}{p-2}} 
 	= \left( \bariint_{ Q_{\tilde{d}(r) r, r^{p}} }  \left( \frac{u}{\tilde{d}(r) r} \right)^{\alpha p} \ddd \right)^{\frac{1}{\alpha p}}   
 	= \lim_{i \to \infty} \left( \bariint_{ Q_{\tilde{d}(r)\rho_i, r^{p}} }  \left( \frac{u}{\tilde{d}(r) \rho_i } \right)^{\alpha p} \ddd \right)^{\frac{1}{\alpha p}} \\
 	\geq \liminf_{i \to \infty} \theta_i^{\frac{n}{\alpha p} + 1} \left( \bariint_{ Q_{\theta_i \tilde{d}(r)\rho_i, r^{p}} }  \left( \frac{u}{\theta_i \tilde{d}(r) \rho_i } \right)^{\alpha p} \ddd \right)^{\frac{1}{\alpha p}}
 	=  \theta^{\frac{n}{\alpha p} - \frac{2}{p-2}}  \tilde{d}(r)  ^{- \frac{p}{p-2}} 
\end{multline*}
which is a contradiction as the exponent of $\theta$ is negative. 
Finally, taking the minimum preserves continuity so the continuity claim on $d_z(r)$ in the second item follows. 
The fact that $d_z(r)$ is monotone increasing is immediate as the definition \eqref{eq:def:def}
is in terms of a minimum.

The first inequality in the third item is an immediate consequence of the second item.
To prove the other bound, 
take $0 < \rho \leq  s \leq R$.  
Denote $\rho_{max}(d(\rho) ) := \tilde{\rho}$ and $\delta := d(\rho) = d(\tilde{\rho})$. 
If $s \leq \tilde{\rho}$, it holds $d(\rho) = d(s) $ and the claimed bound is trivially satisfied.
Hence we can assume $s > \tilde{\rho}$.
We recall that by \eqref{eq:rhomaxprop}
equation \eqref{intrinsic:add:1} holds now with equality $L(z,\tilde{\rho}, \delta) = R(\delta)$.
Then by the definitions
\begin{multline*}
\delta^{\beta} 	= \frac{1}{\lambda} \left( \frac{1}{|Q_{\tilde{\rho},\tilde{\rho}^{p}}|} \iint_{Q_{\tilde{\delta} \tilde{\rho}, \tilde{\rho}^{p}}}  \left( \frac{u}{\tilde{\rho}} \right)^{\alpha p} \ddd \right)^{\frac{1}{\alpha p}} \\
				\leq  \frac{1}{\lambda} \frac{s}{\tilde{\rho}} \left( \frac{|Q_{s,s^{p}}|}{|Q_{\tilde{\rho},\tilde{\rho}^{p}}|} \right)^{\frac{1}{\alpha p}} \left( \frac{1}{|Q_{s,s^{p}}|} \iint_{Q_{d(s)s, s^{p}}}  \left( \frac{u}{s} \right)^{\alpha p} \ddd \right)^{\frac{1}{\alpha p}}
				\leq \left( \frac{s}{\tilde{ \rho}} \right)^{\frac{n+p(\alpha + 1)}{\alpha p}} d(s)^{\beta}
\end{multline*}
so that using $\beta < 0$, $\rho \le \tilde{\rho}$ and the fact that $d_z(\cdot)$ is increasing
we see
\[d(s) \leq \left( \frac{s}{\rho} \right)^{\frac{(n+p \alpha + p)(p-2)}{2 \alpha p - n(p-2)}} d( \rho) .\]
This concludes the proof of the third item.

It remains to prove the last two items in the list. 
Consider two points $z$ and $w$ and a length $r$ so that $S(z, r ) \cap S(w,r) \neq \varnothing$. 
We want to show $S(z,r) \subset S(w, c_1 r )$. 
This is trivially true for $c_1 = 3$ if $d_z(r) \leq d_w(r)$. 
Let $N \ge 1$.
If $r \in [R/N,2R] $, it holds
\begin{align*}  \left( \bariint_{ Q_{\delta r,r^{p}}(w) }  \left( \frac{u}{ \delta r} \right)^{\alpha p} \ddd \right)^{\frac{1}{\alpha p}}
	 &\leq \left( \frac{|Q_{4R,(4R)^{p}}|}{| Q_{\delta r,r^{p}}|} \right)^{\frac{1}{\alpha p}} \frac{ N \lambda_0}{\delta} \\
	&\le 4^{-1}(4N)^{\frac{n+p}{\alpha p} + 1}  \delta^{-\frac{n}{\alpha p}-1} \lambda_0 
	\le \frac{(4N)^{\frac{n+p}{\alpha p} + 1}}{4C_0} \lambda  \delta^{-\frac{n}{\alpha p}-1}  
	\end{align*}
Hence, if we choose $C_0 \ge 4^{-1}(4N)^{\frac{n+p}{\alpha p} + 1}$, 
then any $\delta $ with
\[ \delta^{-\frac{n}{\alpha p} - 1} \leq   \delta^{- \frac{p}{p-2}} \]
satisfies $\delta \leq d_w(r)$. 
In particular because $\alpha > (n/2)(1-2/p)$,
this is the case for all $\delta \in (0,1]$. 
This proves the fourth item of the claim as a side product.
When $r \ge R/10$,
we can also use this with $N = 10$ to conclude $d_w(r) = 1 \geq d_z(r)$
so that the claim in the fifth item follows with $c_1 = 3$.

It remains to deal with the case when $r < R / 10$ and $d_z(r) > d_w(r)$. 
For a moment, we denote
\begin{align*}
\delta_z = d_z(r), \quad
\delta_w = d_w(r), \quad
\rho_w = \rho_{max}(w, \delta_w)
\end{align*}
Because $\delta_z > \delta_w$ and $\rho_w \geq r$, it follows 
\[ Q_{ 3 \delta_z \rho_w    ,  ( 3 \rho_w )^{p}  }(z) \supset Q_{ \delta_w \rho_w, \rho_w^{p} }(w) .\]
Then by monotonicity of $d_z(\cdot)$ and the definitions
\begin{multline*}
\delta_z^{ \frac{n}{\alpha p}- \frac{2}{p-2}} 
	\geq d_z (3  \rho_w  )^{ \frac{n}{\alpha p}- \frac{2}{p-2} } 
	\geq \frac{1}{\lambda} \left( \frac{1}{|Q_{ 3 \rho_w , ( 3  \rho_w)^{p}}|} \iint_{Q_{ 3 \delta_z \rho_w    ,  (3 \rho_w) ^{p}}(z)}  \left( \frac{u}{ 3 \rho_w  } \right)^{\alpha p} \ddd  \right)^{\frac{1}{\alpha p}} \\
	\geq \frac{1}{\lambda} \left( \frac{1}{|Q_{ 3 \rho_w , ( 3  \rho_w)^{p}}|} \iint_{Q_{  \delta_w \rho_w    ,   \rho_w ^{p}}(w)}  \left( \frac{u}{  3 \rho_w} \right)^{\alpha p} \ddd  \right)^{\frac{1}{\alpha p}} 
	= 3^{-\frac{n+p}{\alpha p}-1} \delta_w^{^{ \frac{n}{\alpha p}- \frac{2}{p-2}}}
\end{multline*} 
so that 
\[
\delta_z \le 3^{- \frac{1}{\beta} \left(  \frac{n+p}{\alpha p}+1 \right)} \delta_w.
\]
Consequently, there exist constants $c_1$ and $c_2$ as claimed.
\end{proof}
The following lemma is a formulation of a Vitali type covering theorem,
and we quote it from \cite{Gianazza2019b}.
The properties established in the previous proposition show that the basis $ \{S(z,r): z \in \mathbb{R}^{n} , \ r > 0\}$ satisfies the assumptions.

\begin{lemma}[Lemma 3.2 in \cite{Gianazza2019b}]
\label{Lm:Vitali:1}
Let
\[
 \{U(x,r) : x \in \Omega,\, r\in (0,R]\}
\] 
be a family of open sets which satisfy the following properties
\begin{enumerate}
\item[(i)] Nestedness: 
\begin{equation*}
\text{ If } \quad  x\in\Omega \quad \text{ and } \quad 0<s<r\le R, \quad \text{then} \quad U(x,s)\subset U(x,r);
\end{equation*}
\item[(ii)] Almost uniform shape: There exists a constant $c_1 > 1$, such that   
\begin{equation*}
\text{if} \quad  U(x,r)\cap U(y,r)\neq \emptyset, \quad \text{then} \quad  U(x,r)\subset U(y,c_1r).
\end{equation*}
\item[(iii)] Doubling property: There exists a constant $a>1$ such that, for all $r\in (0,R]$,
\begin{equation*}
0<\abs{U(x,2r)}\leq a\abs{U(x,r)}<\infty.
\end{equation*}
\end{enumerate}
Then we can find a countable and disjoint subfamily 
$\{ U_i \}$, such that
\[
 \bigcup_{x\in\Omega} U(x,r_x) \subset \bigcup_{i} \tilde{U}_i,
 \]
 where $\tilde{U}_i=U(x_i,2c_1 r_{x_i})$, $|U_i|\sim |\tilde{U}_i|$
 and
 \[
|\Omega | \leq c\sum_i\abs{U_i},
 \]
where the constant $c>1$ depends only on $c_1$, $a$, and the dimension $M$.
\end{lemma}

\section{A Gehring type argument}

To conclude the Gehring lemma for the gradient of the solution, 
we study coverings of its level sets. 
Although the construction of $S(z,r)$ only gave us sub-intrinsic cylinders, 
we can extract some additional information from the actual stopping time construction.
The stopping cylinders from the sub-intrinsic basis will turn out to be intrinsic 
in the sense of \eqref{eq:modtousual_1} and \eqref{eq:modtousual_2}.

Fix a center point in the space time,
let $R > 0$ and fix
\begin{equation}
\label{corr:alphaend}
\alpha > \max \left \lbrace  \frac{n}{2} \left(1- \frac{2}{p}\right) , 1+ \frac{p}{n} \right \rbrace .
\end{equation}
The lower bound $1+p/n$ is related to use of (iv) of Proposition \ref{prop:properties_deformation},
and it is likely that it can be lowered by replacing the equation $d_z(\rho) = 1$ there 
by a smaller lower bound on $d_z(\rho)$ following the reasoning of \cite{Boegelein2018}.
However, as the particular choice of $\alpha$ only plays a qualitative role in our argument, 
due to Proposition \ref{Prop:rhi_intrinsic},
we do not attempt to optimize its exact value.

Fix a cylinder $Q_{4R,(4R)^{p}}$ and let 
\[\lambda_0 = 1 + \left( \bariint_{Q_{4R,(4R)^{p}}}   \left( \frac{u}{4R}\right)^{\alpha p} \ddd \right)^{\frac{1}{\alpha p}}  + \left( \bariint_{Q_{4R,(4R)^{p}}}   |\nabla u|^{p}   \ddd \right)^{\frac{1}{p}}. \]
In what follows,
we consider sets $S(z,\rho)$ constructed using a value of $\alpha$ as in \eqref{corr:alphaend}.
Fix $R\leq R_1 < R_2 \leq 2R$ and denote 
\begin{equation}
\label{eq:notationEset}
E (r, \lambda ) = Q_{r,r^{p}} \cap \{ |\nabla u| > \lambda \} \cap \{ \text{Lebesgue points of} \ |\nabla u| \}
\end{equation} 
for $r \in (0,2R)$. 
 
\begin{lemma}
\label{lemma:stopping}
There exist constants $D = D(n,p,\Lambda)$ and $K = K(n,p,\Lambda,D)$ such that the following holds.
Let $u$ be a non-negative weak solution to \eqref{EQ1} and \eqref{structure}.
Let $C_0 =  [4 D R / (R_2 - R_1) ]^{1+ n/p}$ and
$\lambda \ge C_0 \lambda_0$.
Then for almost every $z \in E(R_1,\lambda)$ there is $\rho_z \in (0,(R_2-R_1)/(2D))$ such that 
\begin{itemize}
	\item it holds for all $\rho  \in [ \rho_z, 2R)$ 
	\[
	 \bariint_{S(z,\rho_z)} | \nabla  u|^{p} \ddd  = \lambda^{p} \ge \bariint_{S(z,\rho)} | \nabla  u|^{p} \ddd ;
	\]
\item every modified cylinder $\tilde{S}(z,\rho_z) := Q_{D d(\rho_z) \rho_z , (D\rho_z)^{p} } $ is $K$-intrinsic with $\alpha = 1$ in the sense of \eqref{eq:modtousual_1} and \eqref{eq:modtousual_2}.
\end{itemize}
\end{lemma}

\begin{proof}
We consider a positive number $D$,
whose exact value will be determined in the course of the proof.
We set $C_0 = (4D R /(R_2-R_1))^{(n+p)/p}$.
Let $z \in E(R_1, \lambda)$ and $\rho \leq R_2- R_1$. 
By the assumption $\lambda \ge C_0 \lambda_0$ and hence 
\begin{equation}
\label{corr:december1}
\bariint_{S(z,\rho)} | \nabla  u|^{p} \ddd \le \frac{(4R)^{n+p}}{d_z(\rho)^{n} \rho^{n+p}} \frac{\lambda^{p}}{C_0^{p}} . 
\end{equation}
Because $\alpha > 1+p/n$, it holds 
\[
\frac{\alpha p}{n+p+\alpha p} \ge \frac{p}{n+p} 
\]
and because $C_0 \ge 1$ it further holds 
\[
C_0^{\frac{\alpha p}{n+p+\alpha p}}  \ge   \frac{4DR}{R_2-R_1}  
\]
so that if $\rho \ge (R_2-R_1)/(2D)$,
then by the fourth item in Proposition \ref{prop:properties_deformation} 
it holds $d_z(\rho) = 1$.
For these values of $\rho$,
the right hand side of \eqref{corr:december1} is bounded by 
\[
4^{n+p} (R/\rho)^{n+p} C_0^{-p} \lambda^{p} \le \lambda^{p}.
\]
As 
\[
 \rho \mapsto \bariint_{S(z,\rho)} | \nabla  u|^{p} \ddd 
\]
is continuous,
we define $\rho_z$ to be the maximal number in $(0,(R_2-R_1)/(2D)]$ such that 
\[
\bariint_{S(z,\rho)} | \nabla  u|^{p} \ddd \ge \lambda^{p} .
\]
By the Lebesgue differentiation theorem such a number exists for almost every $z \in E(R_1,\lambda)$.
By continuity, the inequality above holds as an equality for the maximal $\rho_z$,
and the reverse inequality holds for all $\rho > \rho_z$ by maximality of $\rho_z$.
This concludes the argument for the first item in the lemma.


To prove the second item, fix $z \in Q_{R_1,R_1^{p}}$ and take the cylinder $ S(z,\rho_z) = Q_{d_z(\rho_z) \rho_z, \rho_z^{p}}$ as constructed above. 
We show first that $Q_{D\rho_z d(\rho_z) , (D \rho_z)^{p}}$ satisfies \eqref{eq:modtousual_1} and \eqref{eq:modtousual_2} 
with $\alpha $ as specified in \eqref{corr:alphaend}. 
We start with \eqref{eq:modtousual_1}. 
Note that for all $s \in [\rho_z, R]$
\begin{equation}
\label{eq:cov:1}
\bariint_{Q_{ s d_z( \rho_z) , s^{p} }} | \nabla u|^{p} \ddd  <  \left( \frac{d_z(s)}{d_z( \rho_z)} \right)^{n} \lambda^{p} \leq \left( \frac{s}{\rho_z} \right)^{ \frac{n(n+p \alpha + p)(p-2)}{2 \alpha p - n(p-2)} } \lambda^{p} ,
\end{equation}
as follows from item (iii) of Proposition \ref{prop:properties_deformation}.
By \eqref{eq:cov:1} with $s = 2D \rho_z$ 
\begin{align}
\label{eq:lambdamax}
\bariint_{Q_{ 2D \rho_z d_z(\rho_z)  , (2D \rho_z) ^{p} }} | \nabla u|^{p} \ddd 
	\lesssim \lambda^{p}
	\lesssim \bariint_{Q_{ 2 \rho_z d_z(\rho_z)   , (2\rho_z) ^{p} }} | \nabla u|^{p} \ddd .
\end{align}
Further, by (ii) and (iii) of Proposition \ref{prop:properties_deformation} $d_z( \rho_z) \sim d_z( 2D \rho_z)$.
This together with the definition of $d_z(\rho_z)$ and equation \eqref{eq:lambdamax} implies
\begin{multline*}
\left( \bariint_{Q_{2D \rho_z d_z( \rho_z)  , (2D \rho_z)^{p}}}  \left( \frac{u}{ d_z( \rho_z)  \rho_z} \right)^{\alpha p} \ddd \right)^{\frac{1}{\alpha p}}
	\lesssim \left( \bariint_{Q_{ 2D \rho_z  d(2D \rho_z)  , (2D \rho_z)^{p} }}  \left( \frac{u}{ d_z( 2 D \rho_z)  \rho_z} \right)^{\alpha p} \ddd \right)^{\frac{1}{\alpha p}} \\
	\lesssim d(2D \rho_z )^{- \frac{p}{p-2}} \lambda 
	\lesssim d_z(\rho_z)^{- \frac{p}{p-2}} \bariint_{Q_{ D \rho_z d_z( \rho_z)  , (D\rho_z) ^{p} }} | \nabla u|^{p} \ddd 
\end{multline*}
which is \eqref{eq:modtousual_1} for some value of $K$ only depending on $n$, $p$, $\alpha$ and $D$. 

Next we verify the other condition \eqref{eq:modtousual_2}.
We abbreviate $d_z(\rho_z) = \delta$. 
If $\delta = 1$, the second alternative in \eqref{eq:modtousual_2} is satisfied and there is nothing to prove.
Assume $\delta < 1$ and consider first the case 
\[\rho_z \leq \rho_{max}(z,d_z(\rho_z) ) \leq D \rho_z .\] 
The function $\rho_{max}$ is the one defined in \eqref{eq:rhomax}.
Then by \eqref{eq:rhomaxprop} and \eqref{eq:lambdamax}
\begin{multline*}
\delta^{-\frac{p}{p-2}} = \frac{1}{\lambda} \left( \bariint_{Q_{\delta \rho_{max}(z, \delta ), \rho_{max}(z,\delta )^{p}}}  \left( \frac{u}{ \delta  \rho_{max}(z,\delta ) } \right)^{\alpha p} \ddd \right)^{\frac{1}{\alpha p}}		\\
	\lesssim    \left( \bariint_{ Q_{D \delta \rho_z, (D \rho_z)^{p}} }  \left( \frac{u}{2D\delta \rho_z} \right)^{\alpha p} \ddd \right)^{\frac{1}{\alpha p}}  \left( \bariint_{ Q_{ 2D \delta  \rho_z,(2D\rho_z)^{p}}}   |\nabla u|^{p} \ddd \right)^{-\frac{1}{ p}}  
\end{multline*}
which is the first alternative in \eqref{eq:modtousual_2}.

We are left with the case $\rho_{max}(z, \delta )  > D \rho_z$,
and we shall show that the second alternative in \eqref{eq:modtousual_2} holds.
Set
\[\rho_{*} = \frac{ \rho_{max}(z, \delta)}{D}\]
so that $\rho_* \in (\rho_z, \rho_{max}(z, \delta)) $. 
Now by \eqref{eq:rhomaxprop}
and Proposition~\ref{Prop:rhi_intrinsic} we find 
\begin{align*}
\delta^{-\frac{p}{p-2}} \lambda 
	&=  \left( \bariint_{Q_{ D \rho_* \delta , (D \rho_*)^{p}}}  \left( \frac{u}{  D \rho_* \delta } \right)^{\alpha p} \ddd \right)^{\frac{1}{\alpha p}} 
	\leq C \left( \bariint_{Q_{2D \rho_* \delta, (2D \rho_*)^{p}}}  \left( \frac{u}{ D  \rho_*  \delta  }  \right)^{ p} \ddd \right)^{\frac{1}{ p}} .
\end{align*}
None of our definitions applies to values of $u$ in the large cylinder above,
but we do know about the values of $\nabla u$.
On the right hand side of the display above,
we estimate by the Poincar\'e inequality 
\begin{multline*}
\left(\bariint_{Q_{2D \rho_* \delta, (2D \rho_*)^{p}}}  \left( \frac{u}{ D  \rho_*  \delta  }  \right)^{ p} \ddd \right)^{\frac{1}{p}} \\
	\le (2D \rho_*)^{-1} \left( \bariint_{Q_{2D \rho_* \delta, (2D \rho_*)^{p}}} |u(t,x) - u_{\{t\} \times B_{   \delta \rho_*}} | ^{ p} \ddd \right)^{\frac{1}{p}}+  \frac{C}{D \rho_* \delta}  \left( \barint_{ I_{(2D \rho_*)^{p}} } |u_{B_{  \delta \rho_*}}| ^{ p} \ddd \right)^{\frac{1}{ p}} \\
	\le  C \left( \bariint_{Q_{2D \rho_* \delta, (2D \rho_*)^{p}}} |\nabla u | ^{ p} \ddd \right)^{1/p} + \frac{C}{D \rho_* \delta}  \left( \barint_{ I_{(2D \rho_*)^{p}} } |u_{B_{  \delta \rho_*}}| ^{ p} \ddd \right)^{\frac{1}{ p}}.
\end{multline*}
By \eqref{eq:lambdamax}, the first term is bounded by $C \lambda$,
which is good.
The second term has a spatial support smaller than previously
but still rather long support in time.
This can be dealt with using the equation.

As $p\ge 2$,
we can use Jensen's inequality to bound the second term by
\begin{multline*}
\frac{C}{D \rho_* \delta}  \left( \barint_{ I_{(2D \rho_*)^{p}} } | (u^{p-1})_{B_{  \delta \rho_*}}| ^{ p'} \ddd \right)^{\frac{1}{ p}}
 \\
 \le \frac{C}{D \rho_* \delta}  \left( \barint_{ I_{(2D \rho_*)^{p}} } | (u^{p-1})_{B_{ \delta \rho_*}} - (u^{p-1})_{Q_{  \delta \rho_*, \rho_*^{p}} }| ^{ p'} \ddd \right)^{\frac{1}{ p}} + \frac{C}{D \rho_* \delta} (u^{p-1})_{Q_{  \delta \rho_*, \rho_*^{p}} }^{\frac{1}{p-1}}.
\end{multline*}
According to Lemma \ref{lemma:slice_average}, H\"older's inequality and the definition of $\rho_{*}$ and $\delta$,
the first term is bounded by 
\[
\frac{C}{D \rho_* \delta}  
\left( \frac{1}{(\delta \rho_{*})^{n+1}}  \iint_{Q_{2\delta \rho_*,(2D \rho_*)^{p}} } |\nabla u|^{p-1} \ddd  \right)^{\frac{1}{p-1}}
\le  C    \delta^{-\frac{p}{p-1}}     \lambda .
\]
Writing $q = p-1$ and $q' = (p-1)/(p-2)$,
we use Young's inequality with $\epsilon$ to bound 
\[
C    \delta^{-\frac{p}{p-1}}\lambda  =  ( \lambda^{\frac{1}{q'}} \delta^{-\frac{p}{p-1}}) \cdot (C\lambda^{\frac{1}{q}})
\le  \epsilon \lambda \delta^{- \frac{p}{p-2}} + C_\epsilon \lambda . 
\]
As the first term can be sent back to the left hand side,
we are done with this expression.
Finally, by H\"older's inequality  
\[
\frac{C}{D \rho_* \delta} (u^{p-1})_{Q_{  \delta \rho_*, \rho_*^{p}} }^{\frac{1}{p-1}}
 \le \frac{C}{D} \left( \bariint_{Q_{ \rho_* \delta ,  \rho_*^{p}}}  \left( \frac{u}{ \rho_* \delta } \right)^{\alpha p} \ddd \right)^{\frac{1}{\alpha p}} 
\leq  \frac{C_2}{D} \delta^{-\frac{p}{p-2}} \lambda .
\]
Altogether we have shown 
\[
\lambda \delta^{-\frac{p}{p-2}} \le C_1 \lambda + \epsilon \lambda \delta^{-\frac{p}{p-2}} + \frac{C_2}{D} \delta^{-\frac{p}{p-2}} \lambda . 
\]
As $C_2$ is independent of $D$,
we can choose $D$ large enough so that $\delta^{-\frac{p}{p-2}} \le C $ for a constant only depending on $n$, $p$, $\alpha$ and $\Lambda$.
This is the second case of \eqref{eq:modtousual_2}.

As we have shown that \eqref{eq:modtousual_1} and \eqref{eq:modtousual_2} 
hold for $Q_{D \delta \rho_z , (D \delta \rho_z)^{p} }$ 
with $\alpha $ as specified in \eqref{corr:alphaend}, 
it follows from Proposition \ref{Prop:rhi_intrinsic} 
and the fact 
\[ \bariint_{Q_{2D \delta \rho_z , (2D \delta \rho_z)^{p} }} |\nabla  u|^{p} \ddd \sim \bariint_{Q_{2 \delta \rho_z , (2 \delta \rho_z)^{p} }} |\nabla  u|^{p} \ddd \] 
due to the maximality of $\rho_z$
that they also hold for $Q_{2D \delta \rho_z , (2D \delta \rho_z)^{p} }$ and $\alpha = 1$. 
\end{proof}

We can now use the following reverse H\"older inequality from \cite{Boegelein2018}
and the rest of the argument is standard. 

\begin{proposition}[Proposition 6.1 in \cite{Boegelein2018}]
Let $Q_{\delta r,r^{p}}$ be such that \eqref{eq:modtousual_1} and \eqref{eq:modtousual_2} hold with $\alpha = 1$. Then
\[\bariint_{Q_{\delta r, r^{p}}} |\nabla u|^{p} \ddd \leq C \left( \bariint_{Q_{2\delta r, (2r)^{p}}} |\nabla u|^{q} \ddd \right)^{\frac{p}{q}}  \]
where $q = \max ( np/(n+2), p-1 )$ and $C = C(n,p,\Lambda,K)$.
\end{proposition}

\begin{proposition}
\label{prop:RHI_levelset}
It holds
\[\iint_{E(R_1, \lambda )} \n ^{p} \ddd \lesssim \iint_{E(R_2, \lambda)} \lambda^{p-q}  \n^{q-p} \ddd \]
where $E(R_2, \lambda)$ is the set from \eqref{eq:notationEset} and $\lambda$ as in Lemma \ref{lemma:stopping}.
\end{proposition}

\begin{proof}
The sets $\tilde{S}(z,r) = Q_{D r d_z(r) , (D r)^{p}}(z)$ are defined as the anisotropic $(D^{p},D)$-dilations of $S(z,r)$ (as in Lemma \ref{lemma:stopping}). The basis $\tilde{S}(z, r)$ satisfies the hypotheses of Lemma \ref{Lm:Vitali:1} as can be seen from Proposition \ref{prop:properties_deformation}. Consider the cover $\{ \tilde{S}(z, \rho_z ) : z \in E(R_1, \lambda) \}$ of the set $E(R_1, \lambda)$ where the cylinders $\tilde{S}(z, \rho_z)$ are the ones from Lemma \ref{lemma:stopping} and the constant $c_1$ is as in Lemma \ref{Lm:Vitali:1}. By Lemma \ref{Lm:Vitali:1}, we can extract a countable collection of points $\{z_i\}$ so that 
\[E(R_1, \lambda) \subset  \bigcup_{i} \tilde{S}(z_i, 2c_1 \rho_i) \subset E(R_2, \lambda), \quad \tilde{S}(z_i, \rho_i  ) \quad \text{are disjoint}.   \] 
Because for any $\epsilon > 0$
\begin{multline*}
\lambda^{p} \leq c \bariint_{\tilde{S}(z_i,  \rho_i)} \n^{p} \ddd
			\leq c \left( \bariint_{\tilde{S}(z_i,  \rho_i )}  \n  ^{q} \ddd \right)^{\frac{p}{q}}\\
			\leq c \epsilon \lambda^{p} + c \lambda^{p-q} \cdot \frac{1}{|\tilde{S}(z_i, \rho_i)|} \iint_{\tilde{S}(z_i,\rho_i) \cap E(R_2, \epsilon \lambda)} \n^{q} \ddd ,
\end{multline*}  
it holds for $\epsilon = (2c)^{-1} $ that 
\[ \bariint_{\tilde{S}(z_i, \rho_i)} \n^{p} \ddd \leq c \lambda^{p} \lesssim   \frac{1}{|\tilde{S}(z_i,\rho_i)|} \iint_{\tilde{S}(z_i,\rho_i) \cap E(R_2, \epsilon \lambda)}  \lambda^{p-q}  \n^{q} \ddd . \]
By the inequality in Lemma \ref{lemma:stopping}, 
it holds
\begin{align*}
\bariint_{\tilde{S}(z_i, 2 c_1 \rho_i)} \n^{p} \ddd \lesssim \lambda^{p} \lesssim \bariint_{\tilde{S}(z_i, \rho_i)} \n^{p} \ddd
\end{align*}
so that 
\begin{multline*}
\int_{E(R_1, \lambda) } \n^{p} \ddd 
	\leq \sum_{i} \iint_{\tilde{S}(z_i, 2 c_1 \rho_i)} \n^{p} \ddd 
	\lesssim \sum_{i} \iint_{\tilde{S}(z_i, \rho_i)} \n^{p} \ddd \\
	\lesssim \sum_{i} \iint_{\tilde{S}(z_i,\rho_i) \cap E(R_2, \epsilon \lambda)}  \lambda^{p-q}  \n^{q} \ddd
	\lesssim \iint_{ E(R_2, \epsilon \lambda)}  \lambda^{p-q}  \n^{q} \ddd  .
\end{multline*}
Because of the (trivial) estimate
\[\int_{ E(R_1, \epsilon \lambda)  \setminus E(R_1, \lambda) } \n^{p} \ddd \leq   \iint_{ E(R_2, \epsilon \lambda)}  \lambda^{p-q}  \n^{q} \ddd     \]
and the fact that $\epsilon$ only depends on the data,
we conclude
\[
\int_{E(R_1, \lambda) } \n^{p} \ddd 
	\lesssim \iint_{ E(R_2, \lambda)}  \lambda^{p-q}  \n^{q} \ddd\]
as was claimed.
\end{proof}

\begin{proof}[Proof of Theorem \ref{thm:degenerate}]
Given the Proposition \ref{prop:RHI_levelset}, the theorem follows by a standard argument for Gehring's lemma. This has been done, for instance, in section 7.6 of \cite{Boegelein2018} starting from the inequality as given in Proposition \ref{prop:RHI_levelset}. The result in \cite{Boegelein2018} comes with an additional constant term $1$ on the right hand side, but that can be removed as the structure of the PDE is invariant under scaling. Indeed, by applying the result to $u(\delta^{-p}t, \delta^{-1}x)$ on $Q_{\delta \rho, (\delta \rho)^{p}}$ and sending $\delta \to 0$ one recovers the scaling invariant version of the estimate.
\end{proof}

\bibliography{references}

\begin{thebibliography}{10}

\bibitem{Boegelein2018}
V.~B\"{o}gelein, F.~Duzaar, J.~Kinnunen, and C.~Scheven.
\newblock Higher integrability for doubly nonlinear parabolic systems.
\newblock {\em J. Math. Pures Appl. (9)}, 143:31--72, 2020.

\bibitem{Boegelein2019}
V.~{B\"ogelein}, F.~{Duzaar}, R.~{Korte}, and C.~{Scheven}.
\newblock The higher integrability of weak solutions of porous medium systems.
\newblock {\em Advances in Nonlinear Analysis}, 8:1004--1034, 2019.

\bibitem{Bojarski1957}
B.~V. Boyarski\u{\i}.
\newblock Generalized solutions of a system of differential equations of first
  order and of elliptic type with discontinuous coefficients.
\newblock {\em Mat. Sb. N.S.}, 43(85):451--503, 1957.

\bibitem{DiBenedetto1993a}
E.~{DiBenedetto}.
\newblock {\em Degenerate parabolic equations.}
\newblock New York, NY: Springer-Verlag, 1993.

\bibitem{DiBenedetto1985a}
E.~{DiBenedetto} and A.~{Friedman}.
\newblock H\"older estimates for non-linear degenerate parabolic systems.
\newblock {\em J. Reine Angew. Math.}, 357:1--22, 1985.

\bibitem{Gehring1973}
F.~W. {Gehring}.
\newblock The {L}\(^p\)-integrability of the partial derivatives of a
  quasiconformal mapping.
\newblock {\em Acta Math.}, 130:265--277, 1973.

\bibitem{Gianazza2019}
U.~{Gianazza} and S.~{Schwarzacher}.
\newblock Self-improving property of degenerate parabolic equations of porous
  medium-type.
\newblock {\em Amer. J. Math.}, 141(2):399--446, 2019.

\bibitem{Gianazza2019b}
U.~Gianazza and S.~Schwarzacher.
\newblock Self-improving property of the fast diffusion equation.
\newblock {\em J. Funct. Anal.}, page 108291, 2019.

\bibitem{Giaquinta1982}
M.~{Giaquinta} and M.~{Struwe}.
\newblock On the partial regularity of weak solutions on nonlinear parabolic
  systems.
\newblock {\em Math. Z.}, 179:437--451, 1982.

\bibitem{Giusti2003}
E.~{Giusti}.
\newblock {\em Direct methods in the calculus of variations.}
\newblock Singapore: World Scientific, 2003.

\bibitem{Hynd2017}
R.~Hynd and E.~Lindgren.
\newblock Large time behavior of solutions of {T}rudinger's equation.
\newblock {\em J. Differential Equations}, 274:188--230, 2021.

\bibitem{Kinnunen2007}
J.~{Kinnunen} and T.~{Kuusi}.
\newblock Local behaviour of solutions to doubly nonlinear parabolic equations.
\newblock {\em Math. Ann.}, 337(3):705--728, 2007.

\bibitem{Kinnunen2000}
J.~{Kinnunen} and J.~L. {Lewis}.
\newblock Higher integrability for parabolic systems of \(p\)-{L}aplacian type.
\newblock {\em Duke Math. J.}, 102(2):253--271, 2000.

\bibitem{Kuusi2012a}
T.~{Kuusi}, R.~{Laleoglu}, J.~{Siljander}, and J.~M. {Urbano}.
\newblock H\"older continuity for {T}rudinger's equation in measure spaces.
\newblock {\em Calc. Var. Partial Differential Equations}, 45(1-2):193--229,
  2012.

\bibitem{Kuusi2012}
T.~{Kuusi}, J.~{Siljander}, and J.~M. {Urbano}.
\newblock Local {H}\"older continuity for doubly nonlinear parabolic equations.
\newblock {\em Indiana Univ. Math. J.}, 61(1):399--430, 2012.

\bibitem{Lindgren2019}
E.~Lindgren and P.~Lindqvist.
\newblock On a comparison principle for {T}rudinger's equation.
\newblock {\em arXiv:1901.03591}, 2019.

\bibitem{Lindqvist1990}
P.~{Lindqvist}.
\newblock On the equation \(div(| \nabla u| ^{p-2}\nabla u)+\lambda | u|
  ^{p-2}u=0\).
\newblock {\em Proc. Amer. Math. Soc.}, 109(1):157--164, 1990.

\bibitem{Lindqvist1992}
P.~{Lindqvist}.
\newblock Addendum to ``{O}n the equation div\((|\nabla u|^{p-2}\nabla
  u)+\lambda| u|^{p-2}u=0\)''.
\newblock {\em Proc. Amer. Math. Soc.}, 116(2):583--584, 1992.

\bibitem{Meyers1963}
N.~{Meyers}.
\newblock An \({L}^ p\)-estimate for the gradient of solutions of second order
  elliptic divergence equations.
\newblock {\em Ann. Sc. Norm. Super. Pisa Cl. Sci. (5)}, 17:189--206, 1963.

\bibitem{Meyers1975}
N.~G. {Meyers} and A.~{Elcrat}.
\newblock Some results on regularity for solutions of non-linear elliptic
  systems and quasi-regular functions.
\newblock {\em Duke Math. J.}, 42:121--136, 1975.

\bibitem{Moring2019}
K.~Moring, C.~Scheven, S.~Schwarzacher, and T.~Singer.
\newblock Global higher integrability of weak solutions of porous medium
  systems.
\newblock {\em Commun. Pure Appl. Anal.}, 19(3):1697--1745, 2020.

\bibitem{Moser1964}
J.~{Moser}.
\newblock A {H}arnack inequality for parabolic differential equations.
\newblock {\em Comm. Pure Appl. Math.}, 17:101--134, 1964.

\bibitem{Schwarzacher2014}
S.~{Schwarzacher}.
\newblock H\"older-{Z}ygmund estimates for degenerate parabolic systems.
\newblock {\em J. Differential Equations}, 256(7):2423--2448, 2014.

\bibitem{Siljander2010}
J.~{Siljander}.
\newblock Boundedness of the gradient for a doubly nonlinear parabolic
  equation.
\newblock {\em J. Math. Anal. Appl.}, 371(1):158--167, 2010.

\bibitem{Trudinger1968}
N.~S. {Trudinger}.
\newblock Pointwise estimates and quasilinear parabolic equations.
\newblock {\em Comm. Pure Appl. Math.}, 21:205--226, 1968.

\bibitem{Vespri1992}
V.~{Vespri}.
\newblock On the local behaviour of solutions of a certain class of doubly
  nonlinear parabolic equations.
\newblock {\em Manuscripta Math.}, 75(1):65--80, 1992.

\end{thebibliography}

\bibliographystyle{abbrv}

\end{document}